\documentclass[11pt,reqno]{amsart}
\usepackage{graphicx}
\usepackage{amsfonts,amsmath,amssymb}
\newcommand{\cx}{{\mathbb{C}}}

\newcommand {\Q}{\mathcal Q}
\newtheorem{thm}{Theorem}[section]
\newtheorem{propos}[thm]{Proposition}
\newtheorem{corol}[thm]{Corollary}
\newtheorem{thm1}{Theorem}
\theoremstyle{definition}
\newtheorem{dfn}[thm]{Definition}
\newtheorem{ex}[thm]{Example}
\newtheorem{rema}[thm]{Remark}
\newcommand{\im}{\ensuremath{\mbox{\rm Im}\,}}
\newcommand{\re}{\ensuremath{\mbox{\rm Re}\,}}

\newcommand{\remark}[1]{\smallskip  \noindent \bf Remark #1.\rm\,\,}

\newcommand{\CC}[1]{\mathbb{C}^{#1}}
\newcommand{\CP}[1]{\mathbb{CP}^{#1}}
\newcommand{\RR}[1]{\mathbb{R}^{#1}}

\newcommand{\z}{\zeta}

\title[Analytic Continuation of Holomorphic Mappings]{Analytic Continuation of Holomorphic Mappings
From Nonminimal Hypersurfaces}
\author {I. Kossovskiy}
\address{Department of Mathematics, The University of Western Ontario, London, Ontario N6A 5B7 Canada}
\email{ikossovs@uwo.ca}
\author {R. Shafikov}
\email{}
\address{Department of Mathematics, The University of Western Ontario, London, Ontario N6A 5B7 Canada}
\email{shafikov@uwo.ca}
\begin{document}



\date{October 3, 2012}

\begin{abstract}
We study the analytic continuation problem for a germ of a
biholomorphic mapping from a nonminimal real hypersurface
$M\subset\CC{n}$ into a real hyperquadric $\mathcal
Q\subset\CP{n}$ and prove that under certain nondegeneracy
conditions any such germ extends locally biholomorphically along
any path lying in the complement $U\setminus X$ of the complex
hypersurface $X$ contained in $M$ for an appropriate neighbourhood
$U\supset X$. Using the monodromy representation for the
multiple-valued mapping obtained by the analytic continuation we
establish a connection between nonminimal real hypersurfaces and
singular complex ODEs.
\end{abstract}

\maketitle

\tableofcontents

\section{Introduction and main results}
Let $H(\zeta,\bar \zeta)$ be a nondegenerate Hermitian form in
$\CC{n+1}$ with $k+1$ positive and $l+1$ negative eigenvalues,
$k+l=n-1,\,0\leq l\leq k\leq n-1$. We call a hypersurface
$\Q\subset\CP{n}$ a {\it $(k,l)$-hyperquadric} if it is given in
homogeneous coordinates by
\begin{equation}\label{e:Q}
\Q = \left\{ [\zeta_0,\dots,\zeta_n]\in \cx\mathbb P^{n} :
H(\zeta,\bar \zeta) =0  \right\}.
\end{equation}
Clearly, $\Q\subset\CP{n}$ is a compact smooth real algebraic Levi
nondegenerate hypersurface with $(k,l)$ being the signature of its
Levi form. In particular, the unit sphere $S^{2n-1}\subset\CC{n}$
is an $(n-1,0)$-hyperquadric.

Let $M$ be a connected smooth real analytic hypersurface in
$\cx^n$, $n>1$. It was shown by S.~Pinchuk for $\mathcal
Q=S^{2n-1}$ \cite{Pi1}, and by D.~Hill and the second author
\cite{HiSh} for the general case that if $M$ is Levi nondegenerate
then a germ of a local biholomorphic map $f: M \to \Q$ extends
locally biholomorphically along any path on $M$ with the extension
sending $M$ to $\Q$. This leads to the following definition: a
Levi nondegenerate hypersurface $M$ is called $(k,l)$-{\it
spherical} at a point $p\in M$ if there exists a germ at $p$ of a
biholomorphic map $f$ sending the germ $(M,p)$ of $M$ at $p$ onto
the germ of a $(k,l)$-hyperquadric $\Q$ at $f(p)$. It follows then
that a Levi nondegenerate hypersurface $M$ is $(k,l)$-spherical at
one point iff it has this property at all points and we simply
call $M$ a {\it $(k,l)$-spherical} hypersurface. Similar extension
result holds if instead of Levi nondegeneracy one assumes that $M$
is {\it essentially finite}, a condition on the so-called Segre
map of $M$ generalizing Levi nondegeneracy, see \cite{Sh} and
\cite{HiSh}. Using arguments similar to those in \cite{HiSh} one
can further generalize Pinchuk's theorem to the case when $M$ is
merely \it minimal \rm in the sense of Tumanov\,\cite{tumanov},
i.e., when $M$ does not contain any germs of complex
hypersurfaces, see \cite{ShVe}.

In this paper we study the analytic continuation phenomenon for
biholomorphic maps from a {\it nonminimal} real-analytic
hypersurface $M$, i.e., when $M$ contains a complex hypersurface
$X$. In this case the Levi form of $M$ vanishes identically on
$X$, and $M$ is not essentially finite at points in $X$. Also note
that nonminimality is equivalent to the {\it infinite type}
condition in the sense of Kohn and Bloom-Graham (see,
e.g.,~\cite{ber}). In appropriate local coordinates $(z,w)
\in\CC{n-1}\times \CC{}$ near the origin a real-analytic
nonminimal hypersurface $M$ is given by
$$
\im w=(\re w)^m\Phi(z,\bar z,\re w),
$$
where $\Phi(z,\bar z,\re w)$ is a real-analytic function
satisfying certain reality condition, $m \geq 1$ is an integer,
and $X=\{w=0\}\subset M$ is the complex hypersurface.

Within the study of CR invariants of real hypersurfaces, the nonminimal case
is considered to be particularly difficult, and very little is known in this setting.
Some motivational examples and partial results concerning automorphism groups
were obtained by V.~Beloshapka~\cite{belnew}, P.~Ebenfelt, B.~Lamel and D.~Zaitsev~\cite{elz},
M.~Kolar and B.~Lamel~\cite{lamel}, see also references therein.  Further, in \cite{lamel}
the authors give normal form in the so-called \it ruled \rm case, i.e., when
the  function $\Phi(z,\bar z,\re w)$ is independent of $\re w$ and $m=1$. This is somewhat
analogous to the \it rigid \rm case for finite type hypersurfaces. However, in the general
case, the $z$- and the $w$-variables in the expansion of $\Phi(z,\bar z,\re w)$ can mix,
which prevents the use of the Chern-Moser-type machinery for the construction of the normal
form. It becomes apparent that different methods should be employed for the study of
nonminimal hypersurfaces. In this paper we use the approach of analytic continuation as a
tool for propagation of local CR invariants of real hypersurfaces and determination of CR
invariants at nonminimal points.

Our results show that the nonminimal case is quite different from the minimal one,
as new geometric phenomena occur. The following illuminating example shows that the analytic
continuation fails in general in the nonminimal case. The hypersurface $M^{\rm log}$ appeared
first in 2001 in the work of V.~Beloshapka and A.~Loboda as an example of an infinite type
hypersurface with a large automorphism group.

\begin{ex}
Consider the real-analytic hypersurface given by
\begin{equation}
M^{\rm log} =\left\{(z,w=u+iv)\in \cx^2 : v=u\,\tan |z|^2,\ |z|<1
\right\},
\end{equation}
or by a global ``complex defining equation" $w=\bar we^{2iz\bar
z}.$ Note that $M^{\rm log}$ contains the complex hypersurface
$X=\{w=0\}$, but it is Levi nondegenerate at all other points. The
set $X$ divides $M^{\rm log}$ into two connected components: $M^+$
given by $\{u>0\}$ and $M^-$ given by $\{u<0\}$. It follows that
\[
M\setminus X = \left\{\arctan \frac{v}{u}=|z|^2\right\},
\]
and so for $u>0$, we have ${\rm Im\,} (\ln w)=|z|^2$. This shows
that the map
 $F:\,\{u>0\}\to \CC{2}$ given by
 $$
 z^*= z,\,w^*=\ln w, \ \ \ -\frac{\pi}{2}<{\rm Arg} \,w<\frac{\pi}{2} ,
 $$
maps $M^+$ onto an open subset of the  nondegenerate hyperquadric
$$
\Q=\{(z^*, w^*)\in \CC{2} : \im w^*=|z^*|^2\}.
$$
However, $F$ clearly does not extend across $X$ (neither
holomorphically nor as a holomorphic
 correspondence). In fact, the branch $F^-:\,\{u <0\} \to \CC{2}$ of the multiple-valued map
 $z^*= z,\,w^*=\ln w$ satisfying $\frac{\pi}{2}<\mbox{Arg}\,w<\frac{3\pi}{2}$ sends $M^-$ into
 an open subset of a {\it different} hyperquadric
 $$
 \mathcal{\tilde Q}=\{(z^*, w^*)\in \CC{2} : \im w^*-\pi =|z^*|^2\}.
 $$
 \end{ex}

 This, and examples given in  \cite{belnew}, \cite{elz}, and \cite{lamel}, as well as other examples given
 later in this paper suggest that given a nonminimal hypersurface $M$ containing a complex hypersurface
 $X$, and a local map $f$ from $M$ into a hyperquadric $\mathcal Q$, one cannot expect in general that $f$
 extends holomorphically to $X$, and since the complement of $X$ is not simply connected, the analytic
 continuation, if exists, can lead to a multiple-valued extension of $f$. Furthermore, since $M\setminus X$ is
 not connected, different components of $M\setminus X$ can be mapped by different branches of the
 multiple-valued extension into different hyperquadrics.

 Our principal result establishes such multiple-valued holomorphic extension phenomenon.
We call a real hypersurface $M$ containing a complex hypersurface
$X$ \it pseudospherical \rm if at least one of the components of
$M\setminus X$ is $(k,l)$\,-\,spherical for some $k+l=n-1,\,0\leq
l\leq k\leq n-1$.

\begin{thm1}\label{t.1}
Let $M\subset\CC{n}$ be a connected smooth real-analytic
hypersurface containing a complex hypersurface $X$. Assume that
$M\setminus X$ is Levi nondegenerate, and that $M$ is
pseudospherical.  Then:

\smallskip

(i)  There exists an open neighbourhood $U$ of $X$ in $\mathbb C^n$ such that for $p\in
(M \setminus X)\cap U$  any biholomorphic map $f$ of $(M,p)$ into
a $(k,l)$\,-\,hyperquadric $\Q$ extends analytically   along any
path in $U\setminus X$ as a locally biholomorphic map into
$\CP{n}$. In particular, $f$ extends to a possibly multiple-valued
locally biholomorphic analytic mapping $U\setminus
X\longrightarrow\CP{n}$ in the sense of Weierstrass.

\smallskip

(ii) If one of the components of $M\setminus X$ is
$(k,l)$\,-\,spherical, then the second component is
$(k',l')$\,-\,spherical with, possibly, $(k,l)\neq(k',l')$.
\end{thm1}

Somewhat surprisingly,  Example~\ref{ex.nonsh} of Section~6 gives
a real hypersurface $M \subset \mathbb C^3$ for which
$(k,l)\neq(k',l')$ in part (ii) of Theorem 1. However, if $f$
extends to $U\setminus X$ as a single-valued map, then both
components of $M\setminus X$ have the same signature of the Levi
form as shown in Proposition~\ref{p.kl}.

The nature of multiplicity of the extension in the above theorem
depends only on the geometry of the hypersurface $M$, and does not
depend on the choice of the map $f$. In Section~7 we give a
precise description of the monodromy of analytic continuation of
$f$ about $X$ by constructing explicitly the monodromy operator,
in analogy with the  corresponding theory of singular ODEs. To
suppress technical details we give a simplified formulation of our
result below and refer the reader to  Section~7 for further
details.

\begin{thm1}\label{t.2}
Under the conditions of Theorem~\ref{t.1}, there exists a linear
representation
$$
\phi : \pi_1 (U\setminus X) \longrightarrow {\rm Aut} (\CP{n})
$$
 such that the analytic continuation $\tilde f$ of $f$ along a
cycle $\gamma\subset U \setminus X$ satisfies
$$
\tilde f = \phi (\gamma) \circ f.
$$
The cyclic subgroup $\phi(\pi_1 (U\setminus X)) \subset {\rm Aut}
(\CP{n})$ is determined by $M$ uniquely up to conjugation, and the
conjugacy class is a biholomorphic invariant of $M$.
\end{thm1}

\noindent{\bf Remark 3.} Additional motivation for the study of mappings from nonminimal
hypersurfaces into quadrics comes from the fact that only the hypersurfaces that are Levi
nondegenerate and $(k,l)$-spherical outside the complex locus $X$ admit "large" automorphism
groups. In particular, for $n=2$ these are the only hypersurfaces that admit the
estimate $\mbox{dim\,Aut}(M,0)\geq 5$ for the local automorphism
group (see \cite{obzor} for details, and \cite{lamel} for some
examples).

\bigskip

The paper is organized as follows. In Section 2 we provide the
necessary background about Segre varieties and the associated
notion of Segre map, and prove the local one-to-one property of
the Segre map in $U\setminus X$ for hypersurfaces under
consideration. In Section 3 we apply this property to study the
behaviour of \it Segre sets \rm and show that \it any \rm point in
the punctured neighbourhood $U\setminus X$ can be connected with
$q\in M\setminus X$ by  a chain of Segre varieties. We use this
fact to construct in Section~4 the desired analytic continuation
along some specific paths by means of extension along Segre
varieties, in the spirit of K.~Diederich and J.E.~Forn\ae ss
\cite{DiFo}, K.~Diederich and S.~Pinchuk \cite{DiPi}, and
\cite{Sh}. We also introduce the notion of $\Q$-Segre property for
a map $f$ and use it for appropriate understanding of the
extension along (iterated) Segre varieties. Combining the results
of Sections~3 and~4, we prove a crucial corollary establishing analytic continuation of
the initial germ $F_0$ to an \it arbitrary \rm
point $r\in U\setminus X$ along some specific path. In Section 5
we use this result to prove the continuation along an arbitrary
path, using the global nature of the (complexified) automorphism
group of the target hyperquadric $\Q$, which gives the first part
of our principal result. We also provide a number of examples of
analytic maps that extend germs of local
biholomorphic mappings to a hyperquadric. Most of these are
certain blow-ups of the unit 3-sphere giving both single-valued and
multiple-valued maps. In Section~6 we prove the second
part of Theorem~\ref{t.1} and give an example of $M$ that can be
mapped to {\it inequivalent} hyperquadrics. In Section~7 we
describe the monodromy of the obtained multiple-valued map showing
that the monodromy can be expressed in terms of a (scaled) element
of $\mbox{GL}_{n+1}(\CC{})$ - the monodromy matrix. We also
establish an intriguing  connection between nonminimal
pseudospherical hypersurfaces and linear differential equations of
order $n$ with an isolated singular point by proving the \it
Monodromy formula \rm for the multiple-valued mapping $F$. The
hypersurface $X\subset M$, playing the role of an isolated
singularity for holomorphic maps under consideration, becomes an
analogue of a single point in $\CP{1}$ as an isolated singularity
of linear differential equations. In Section~8 we consider
separately the case where $M$ is algebraic and prove that the
multiple-valued mapping $F$ in this case extends to $X$, either
holomorphically or as a holomorphic correspondence.

\medskip

{\bf Acknowledgments.} We would like to thank V.~Beloshapka and
B.~Lamel for fruitful discussions that motivated in many aspects
the current research project. The second author is partially
supported by the Natural Sciences and Engineering Research Council
of Canada.

\section{Background: Segre varieties}\label{s:SV}

Let $M$ be a smooth real analytic hypersurface in $\cx^n$, $n\geq
2$, $0\in M$, and $U$ a neighbourhood of the origin. In what
follows in this paper we consider only connected real
hypersurfaces. If $U$ is sufficiently small then $M\cap U$ admits
a real analytic defining function $\rho(Z,\overline Z)$, where the
function $\rho(Z,W)$ is holomorphic in $U\times \overline U$, and
for every point $\zeta\in U$ we can associate to $M$ its so-called
Segre variety in $U$ defined as
\begin{equation}
Q_\zeta= \{Z\in U : \rho(Z,\overline \zeta)=0\}.
\end{equation}
Note that Segre varieties depend holomorphically on the variable
$\overline \zeta$. In fact, we may find  a suitable pair of
neighbourhoods $U_2={\ U_2^z}\times U_2^w\subset \cx^{n-1}\times
\cx$ and $U_1 \Subset U_2$ such that
\begin{equation}\label{product}
  Q_\zeta=\left \{(z,w)\in U^z_2 \times U^w_2: w = h(z,\overline \zeta)\right\}, \ \ \zeta\in U_1,
\end{equation}
is a closed complex analytic subset. Here $h$ is a holomorphic
function. Following \cite{DiPi} we call $U_1, U_2$ a {\it standard
pair of neighbourhoods} of the origin. From the definition and the
reality condition on the defining function the following basic
properties of Segre varieties easily follow:
\begin{equation}\label{segre1}
  Z\in Q_\zeta \ \Leftrightarrow \ \zeta\in Q_Z,\quad Z,\zeta\in U_1,
\end{equation}
\begin{equation}\label{segre2}
  Z\in Q_Z \ \Leftrightarrow \ Z\in M, \quad Z,\zeta\in U_1,
\end{equation}
\begin{equation}\label{segre3}
  \zeta\in M \Leftrightarrow \{Z\in U_1: Q_\zeta=Q_Z\}\subset M,\quad\zeta\in U_1.
\end{equation}
The set $I_\zeta:=\{Z\in U_1: Q_\zeta=Q_Z\}$ is also a complex
analytic subset of $U_1$. If $M$ contains a complex hypersurface
$X$, then for any $p\in X$, $Q_p = X$ and $Q_p\cap
X\neq\emptyset\Leftrightarrow p\in X$, so $I_p = X$.

If $f: U \to U'$ is a holomorphic map sending a smooth real
analytic hypersurface $M\subset U$ into another such hypersurface
$M'\subset U'$, and $U$ is as in \eqref{product}, then
\begin{equation}\label{e:inv}
f(Z)=Z' \ \ \Longrightarrow \ \ f(Q_Z)\subset Q'_{Z'},
\end{equation}
for $Z$ close to the origin. The invariance property of Segre
varieties will play a fundamental role in our arguments. For the
proofs of these and other properties of Segre varieties see, e.g.,
\cite{DiFo}, \cite{DiPi}, \cite{DiPi2}, or \cite{ber}.

The space of Segre varieties $\{Q_Z : Z\in U_1\}$ can be
identified with a subset of $\cx^N$ for some $N>0$ in such a way
that the so-called Segre map $\lambda : Z \to Q_Z$ is holomorphic
(see \cite{DiFo}). Since we have $Q_p=X$ for any $p\in X$, the
Segre map $\lambda$ sends the entire $X$ to a unique point in
$\CC{N}$ and, accordingly, $\lambda$ is not even finite-to-one
near each $p\in X$ (i.e., $M$ is \it not essentially finite \rm at
points $p\in X$). On the other hand, if $M$ is Levi nondegenerate
at a point $p$, then its Segre map is one-to-one in a
neighbourhood of $p$. In fact, the last property can be
strengthened as follows.

\begin{propos}\label{p.21}
Let $M\subset\CC{n}$ be a smooth real-analytic hypersurface,
containing a complex hypersurface $X$, $0\in X\subset M$. Suppose
that $M\setminus X$ is  Levi nondegenerate. Then a standard pair
of neighbourhoods $(U_1,U_2)$ for $ 0 \in M$ can be chosen in such
a way that the Segre map $\lambda:\,U_1 \longrightarrow \CC{N}$ is
locally injective at any point $p\in U_1\setminus X$.
\end{propos}

\begin{proof}
Denote by $\Sigma$ the set of points where the rank of the map
$\lambda$ is less than $n$.  Clearly, $\Sigma$ is a
complex-analytic subset of $U_1$, and $X\subset\Sigma$. We will
show that $U_1$ can be taken sufficiently small so that $\Sigma
\cap U_1 = X\cap U_1$. Let $\tilde \Sigma$ be any irreducible
component of $\Sigma$ of positive dimension such that $0 \in
\tilde \Sigma$. It follows from injectivity of $\lambda$ at Levi
nondegenerate points that $\tilde\Sigma\cap M\subset X$.

Let $U^+$ and $U^-$ be the two connected components of
$U_1\setminus M$. We claim that either $\tilde\Sigma\subset
\overline{U^+}$ or $\tilde\Sigma\subset \overline{U^-}$. Indeed,
suppose that on the contrary, $\tilde \Sigma\cap U^+\neq
\emptyset$ and $\tilde \Sigma\cap U^-\neq \emptyset$. Let
$d:=\mbox{dim}\,\tilde\Sigma$. First observe that
$\tilde\Sigma\cap M \not\subset\tilde\Sigma^{\mbox{sing}}$, for
otherwise the set $\tilde\Sigma^{\mbox{sing}}$ would divide
$\tilde\Sigma^{\mbox{reg}}$ into a union of two open components
(because $M$ divides $U_1$, and therefore $\tilde\Sigma\cap M$
divides $\tilde\Sigma$). This is however impossible because for
irreducible $\tilde \Sigma$, the set $\tilde\Sigma^{\mbox{reg}}$
is connected (see, e.g., \cite{Ch}). It follows that
$\tilde\Sigma\cap M$ contains regular points of $\tilde\Sigma$,
and, considering a small neighbourhood of any such point, we
conclude that the dimension of the real-analytic set
$\tilde\Sigma\cap M$ equals $2d-1$ (since this set splits
$\tilde\Sigma^{\mbox{reg}}$). On the other hand, $\tilde\Sigma\cap
X\subset\tilde\Sigma\cap M\subset\tilde\Sigma\cap X$ from the
above arguments, so $\tilde\Sigma\cap M=\tilde\Sigma\cap X$, which
shows that the dimension of $\tilde\Sigma\cap M$ can not be odd.
That proves the claim.

Now if, for example, $\tilde\Sigma\subset \overline{U^+}$, we can
move $\tilde\Sigma$ along the normal direction to $M$ at $0$ and
get $\tilde\Sigma\cap W\subset U^+$ for the perturbed set
$\tilde\Sigma$ and a sufficiently small neighbourhood $W$ of the
origin. This means that $\tilde\Sigma\cap X\neq\emptyset,$ though
for the perturbed set $\tilde\Sigma\cap X=\emptyset$, which is a
contradiction, because $\dim \tilde\Sigma + \dim X \ge n$ and
therefore their intersection is stable under small perturbations
(\cite{Ch}).

From the above we conclude  that all components of $\Sigma$,
different from $X$, do not intersect $X$. The zero-dimensional
components of $\Sigma$ do not accumulate at 0, and therefore, we
may choose the neighbourhood $U_1$ so small that $\Sigma=X\cap
U_1$, as required.
\end{proof}\rm

Proposition~\ref{p.21} motivates the following

\begin{dfn}
A smooth real-analytic hypersurface $M$, containing a complex
hypersurface $X\ni 0$, is called \it Segre-regular in a
neighbourhood $U$ of the origin, \rm if the Segre map $\lambda$ of
$M$ is locally injective at each point $p\in U\setminus X$.
\end{dfn}

We immediately conclude from Proposition 2.1, that \it for a
smooth real-analytic hypersurface $M$, containing a complex
hypersurface $X\ni 0$ and Levi nondegenerate in $U\setminus X$, a
standard pair of neighbourhoods $U_1,U_2$ of the origin can be
chosen in such a way that $M$ is Segre-regular in $U_1$. \rm

The Segre-regularity will be the basic assumption for most of the
statements in this paper. We note, once again, that for a
Segre-regular in a neighbourhood $U$ hypersurface the image
$\lambda(X)$ consists of a unique point in $\cx^{N}$ and near all
points $p\in U\setminus X$ the map $\lambda$ is one-to-one.

Finally we describe the geometry of Segre varieties for the
nondegenerate hyperquadric $\Q$ in the target domain. In this
case, the Segre variety of a point $[\zeta_0,...,\zeta_n]
\in\CP{n}$ is the projective hyperplane
$$
Q'_\zeta = \{[\xi_0,...,\xi_n]\in \cx\mathbb P^n:
H(\xi,\bar\zeta)=0\},
$$
and the set $\{Q'_\zeta,\,\zeta\in\CP{n}\}$ of all Segre varieties
coincides with the space $(\CP{n})^*$ of all projective
hyperplanes in $\CP{n}$. The Segre map $\lambda'$ in this case is
a global natural one-to-one correspondence between  $\CP{n}$ and
the space $(\CP{n})^*$.

\section{Exhaustion of a punctured neighbourhood by Segre sets}

Let $M,X,\,U_1,\,U_2$ be as in Section 2. Following \cite{ber}, we
introduce the \it Segre sets \rm of $M$ in a neighbourhood of the
origin. We choose $q\in U_1$ and define the zero and the first
Segre sets $S^q_0,\,S^q_1$ of $q$ simply as $S^q_0:=\{q\}$ and
$S^q_1:=Q_q\cap U_1$. Higher order Segre sets $S^q_j,\,j\geq 2$
are defined by induction as
$$
S^q_j:=\left(\bigcup\limits_{r\in S^q_{j-1}}Q_r\right)\cap U_1.
$$
Finally, we define $S^q_\infty:=\bigcup\limits_{j\geq 0}
S^q_{2j}$. For $q\in X$ we have $S^q_j=X\cap U_1$ for any $j\geq
0$. As it is shown in \cite{ber}, Segre sets have the following
properties.

\begin{enumerate}
\item[(a)] $S^q_j\subset S^q_{j+2}$ for $q\in U_1$ and $S^q_j\subset
S^q_{j+1}$ for $q\in M\cap U_1$.
\item[(b)] $r\in S^q_j\Leftrightarrow q\in S^r_j$ and so $r\in
S^q_\infty\Leftrightarrow q\in S^r_\infty$.
\item[(c)] $S^q_j$ can be presented as $\pi (\sigma^q_j)$, where
$\sigma^q_j\subset\CC{N}$ is a complex submanifold ($N>n$), and
$\pi :\,\CC{N}\longrightarrow\CC{n}$ is a holomorphic projection.
\end{enumerate}

In this section we show that the open connected set $U_1\setminus
X$ can be exhausted by the even Segre sets $\{S^p_{2j}\}_{j\geq
1}$ for any fixed $p\in U\setminus X$.

\begin{propos}
Let $M\subset\CC{n}$ be a smooth real-analytic hypersurface,
containing a complex hypersurface $X\ni 0$, and $U_1,U_2$ be the
standard pair of neighbourhoods for $M$ at the origin. Suppose
that $M$ is Segre-regular in $U_1$. Then for any $q\in
U_1\setminus X$,
$$
S^q_\infty=U_1\setminus X.
$$
\end{propos}

\begin{proof}
Property (b) above shows that for any two Segre sets
$S^q_\infty,\,S^r_\infty,\,q,r\in U_1$ either
$S^q_\infty=S^r_\infty$ or $S^q_\infty\cap S^r_\infty=\emptyset$
holds. So $U_1\setminus X$ can be represented as a disjoint union
of some $S^q_\infty,\,q\in U_1\setminus X$ (since each $q\in
S^q_2$). The proposition now asserts that, in fact, this disjoint
union consists of just one element $S^q_\infty$.

We first claim that every $S^q_2,\,q\in U_1\setminus X$, is open
at any point $r\in S^q_2$,  sufficiently close to $q$ except,
possibly, the point $r=q$. Indeed, let $U(q)$ be a neighbourhood
of $q$ such that the Segre map $\lambda$ is one-to-one in $U(q)$.
Take any point $r\in S^q_2$ so that $r\neq q$ and $r\in U(q)$.
Then $r\in Q_s,\,s\in Q_q \cap Q_r$, in particular, $Q_r\cap
Q_q\neq\emptyset$. The injectivity of $\lambda$ in $U(q)$ implies
that $Q_r\neq Q_q$. A sufficiently small perturbation of $r$ does
not change the properties $Q_r\neq Q_q$ (from the definition of
$U(q)$) and $Q_r\cap Q_q\neq\emptyset$ (as in the proof of
Proposition 2.1, we use the fact that the sum of the dimensions of
these two analytic sets is at least $n$ and refer to \cite{Ch}).
So for any $r'$, sufficiently close to $r$, there exists a point
$s'$ such that $s'\in Q_q$ and $s'\in Q_{r'}$, so that $r'\in
Q_{s'}$ and $r'\in S^q_2$, as required. This proves the claim.

Now take any $S^q_\infty,\,q\in U_1\setminus X$ and consider an
interior point $q'\in S^q_2$. Take a ball $B$, centred at $q'$
and such that $B\subset S^q_2$. Then for all $r$ sufficiently
close to $q$ we have $S^r_2\cap B\neq \emptyset$ (this follows
from the continuity of the map $\lambda:\, z\longrightarrow Q_z$).
Therefore, there exists $r'\in B$ such that $r'\in S^r_2$. We get
the inclusions $r\in S^2_{r'},\,r'\in S^2_q$ which imply, by
definition of Segre sets, the inclusion $r\in S^q_4$ for all $r$
sufficiently close to $q$. This shows that $q$ is an interior
point of $S^q_4$.

Taking a point $s\in S^q_{2j}$ for some $j\geq 0$, we use a
similar argument to conclude from openness of $S^q_4$ at $q$, that
$s$ is an interior point of $S^q_{2j+4}$. This finally shows that
all points of the set $S^q_\infty$ are in fact its interior points
and so $S^q_\infty$ is an open set. The connectivity of
$U_1\setminus X$ now implies that
 the decomposition of $U_1\setminus X$ into Segre sets consists of a unique element and
$U_1\setminus X=S^q_\infty$ for any $q\in U_1\setminus X$,  as
required.
\end{proof}

\begin{ex}
For the hypersurface $M^{\rm log}$ (see Introduction) we may
choose $U_1=U_2=\CC{2}$ and $p=(0,1)\in M$. Then $M$ is
Segre-regular in $U_1$ and simple computations show:
$$S^p_0=p,\,S^p_1=\{w=1\},\,S^p_2=(\CC{2}\setminus
X)\setminus\{z=0,\,w\neq 1\},\,S^p_3=S^p_4=...=\CC{2}\setminus X.
$$
It is also not difficult to see, that taking
$U_1=\{|z|<\varepsilon,\,|w|<\varepsilon\}$ and
$p=(\frac{\varepsilon}{2},0)$, all points, lying in the $j$-th
Segre set $S^p_j$, satisfy: $|w|\geq \frac{1}{2}\varepsilon
e^{-2j\varepsilon^2}$. This inequality shows that no Segre set of
a fixed "depth" $j$ can a priori exhaust the punctured
neighbourhood $U_1\setminus X$ for a nonminimal Segre-regular
hypersurface $M$.
\end{ex}

\section{Extension along Segre varieties}

The result of the previous section, showing that iterated Segre
varieties of a fixed point $p\in M\setminus X$ exhaust the
punctured neighbourhood $U_1\setminus X$, suggests that  the
desired continuation of a given local biholomorphic map $F$ of $M$
into a quadric $\Q$ can be obtained by extending $F$ along
iterated Segre varieties of the point $p$. The extension along
Segre varieties is based on their invariance property
\eqref{e:inv} and gives an effective way of holomorphic
continuation for holomorphic maps of real submanifolds in complex
spaces (see \cite{Sh}, \cite{HiSh}).

Let $M,X,U_1,U_2$ be as in Section 2, with $0\in X \subset M$. Let
$p\in (M\setminus X)\cap U_1$. We first introduce the following
notation: by $Q_{p_0,p_1,...p_{j-1}}$ we denote the Segre variety
$Q_{p_{j-1}}$, where $p_0:=p,\,p_k\in Q_{p_{k-1}},\,k=1,2,...,j-1$
so that $p_k\in S^p_k,\,k=0,1,2,...,j-1$ and
$Q_{p_0,p_1,...p_{j-1}}\subset S^p_j$. In this section we show
that a local biholomorphic map $F$, sending the germ of $M$ at a
point $p\in M\setminus X$ into a hyperquadric, can be extended, in
a certain sense, to a neighbourhood of any
$Q_{p_0,p_1,...p_{j-1}}$. For $j=1$ the Segre variety $Q_{p_0}$
contains $p$ and the extension can be understood naturally, while
for $j\geq 2$ the meaning of extension will be specified.

Let $r\in U_1\setminus X$, $U(r)\subset U_1\setminus X$ be an open
polydisc, centred at $r$, and $F:\,U(r)\longrightarrow\CP{n}$ be
a biholomorphic map onto its image. For $q\in U(r)$ and $s\in Q_q$
so that $q\in Q_s$, we denote by $(Q_s)^c$ the connected component
of $Q_s\cap U(r)$, containing $q$. We say that {\it $F$ has the
$\Q$-Segre property in $U(r)$}, if for any $q\in U(r)$ and $s\in
Q_q$ the image $F((Q_s)^c)$ is contained in a projective
hyperplane in $\CP{n}$. The definition is motivated by the fact
that quadrics are the only hypersurfaces in $\CP{n}$ for which the
Segre varieties are projective hyperplanes.

We now formulate the key claim for the main result of this paper.

\begin{propos}
Let $M\subset \CC{n}$ be a smooth real-analytic hypersurface,
$X\subset\CC{n}$ be a complex hypersurface, $0\in X\subset M$,
$(U_1,U_2)$ be a standard pair of neighbourhoods at the origin.
Suppose that $M$ is Segre-regular in $U_1$ and that for some point
$p\in M\setminus X$ and a polydisc $U(p)$, centred at $p$, there
exists a biholomorphic map $F_0:\,U(p)\longrightarrow\CP{n}$ such
that $F_0(M\cap U(p))\subset\Q $, where $\Q\subset\CP{n}$ is a
nondegenerate real hyperquadric. Then there exist connected
neighbourhoods $W_1,W_2,...W_j$ of
$Q_{p_0},Q_{p_1},...,Q_{p_{j-1}}$ respectively (so that $p_k\in
W_k,\,k=0,1,...,j-1$) and locally biholomorphic maps
$$
F_1,F_2,...,F_j,\,\ \ F_k:\,W_k\longrightarrow\CP{n}, \ \
k=1,2,...,j,
$$
such that:

\begin{enumerate}

\item[(i)] The intersection $U(p)\cap W_1$ contains a polydisc $W_0$, centred at $p_0$, such that $F_1$
is a holomorphic extension of $F_0|_{W_0}$.

\item[(ii)] For each $k=2,..,j$ the intersection $W_{k-2}\cap W_k$ contains a polydisc $U (p_{k-2})$,
centred at $p_{k-2}$,  such that $F_k$ is a holomorphic extension
of $F_{k-2}|_{U(p_{k-2})}$.

\item[(iii)]  For each $r\in Q_{p_k},\,k=0,1,...,j-1$ there exists a polydisc $U(r)\subset W_{k+1}$ such that
$F_{k+1}|_{U(r)}$ has the $\Q$-Segre property in $U(r)$.

\end{enumerate}
\end{propos}

\begin{proof} We use the coordinate system in the preimage in the form $(z,w) \in \cx^{n-1}\times \cx$
and  denote by $U^z$ and $U^w$ the projections of a polydisc
$U\subset U_{1}\subset \cx^n$ onto the $z$-coordinate subspace and
the $w$-axis respectively. We also suppose that in these
coordinates Segre varieties of $M$ are graphs of the form
$w=h(z),\,h\in\mathcal{O}(U^z)$ and $X$ is given by $\{w=0\}$. In
the target domain we denote by $Q'_{\zeta}$ the Segre varieties of
points $\zeta \in\CP{n}$ with respect to the hyperquadric $\Q$.

\smallskip

{\bf Step 1.} We first prove part (i) and (iii) for $k=0$. We
choose $W_0\subset U(p)$ to be a polydisc, centred at $p$ and
such that for each Segre variety $Q_q,\,q\in Q_p\cap U_1$ the
intersection $Q_q\cap W_0$ is the graph of a function over
$W^z_0$, in particular, it is connected (the existence of such a
polydisc follows from the fact that $p\in Q_p,\,p\in Q_q$ for
$q\in Q_p$ and that the 1-jets of Segre varieties $Q_q,\,q\in Q_p$
at $p$ are bounded in the intersection of $Q_p$ with the closed
polydisc $\overline{U_1}\subset U_2$). So we can choose a
connected neighbourhood $W_1\supset Q_p$ such that for $s\in W_1$
the intersection $Q_s\cap W_0$ is also connected and nonempty.

We follow the strategy in \cite{DiPi} and \cite{Sh}, and consider
the set
$$
A_1=\{(Z,\zeta)\in W_1\times\CP{n}:\,F_0(Q_Z\cap W_0)\subset
Q'_{\zeta}\}.
$$
In the same way as it is done in Proposition 3.1 in \cite{Sh}, one
can show that $A_1$ is a nonempty closed complex-analytic subset
in $W_1\times\CP{n}$ of dimension $n$. But unlike the situation in
\cite{Sh}, we do not need to exclude an analytically constructible
set from $Q_p$ since the hypersurface in the target domain is the
hyperquadric $\Q$ whose Segre map is globally injective. If for
some $Z\in W_1,\,\z_1,\z_2\in\CP{n}$ and $\z_1\neq \z_2$, then
$$
(Z,\zeta_1),\,(Z,\z_2)\in A_1 \ \ \Longrightarrow \ \ F_0(Q_Z\cap
W_0)\subset Q_{\z_1}\cap Q_{\zeta_2},
$$
which is not possible since $F_0$ is biholomorphic in $W_0$, $\dim
Q_{\z_1}\cap Q_{\z_2} = n-2$ while $Q_Z\cap W_0$ is of dimension
$n-1$. So $A_1$ is, in fact, the graph of a holomorphic map
$F_1:\,W_1\longrightarrow\CP{n}$. To show that $F_1$ is locally
biholomorphic we observe that local injectivity of $\lambda$
implies that for distinct $Z_1,Z_2$ that are close to each other,
the intersection $Q_{Z_1}\cap Q_{Z_2}\cap W_0$ has dimension at
most $n-2$, so  by shrinking $W_1$, if necessary, we conclude that
$(Z_1,\zeta)\in A_1$ and $(Z_2,\zeta)\in A_1$ forces $Z_1=Z_2$, so
that $F_1$ is locally one-to-one and hence biholomorphic. Also,
the invariance property of Segre varieties~\eqref{e:inv} implies
that for $Z\in W_1$, sufficiently close to $p$, $F_0(Z)=F_1(Z)$,
which proves (i).

For the proof of (iii) for $k=0$, we consider the set $V_1$ of
points $q\in U_1$ such that $Q_q\cap W_0\neq\emptyset$. Clearly,
\begin{equation}\label{e:v1}
V_1=\bigcup\limits_{s\in W_0}Q_s.
\end{equation}
Since $W_0$ is open, $V_1$ is also open, and because each $Q_s$ is
path-connected and $W_0$ is open and path-connected, $V_1$ is also
connected. For points $s\in U_1$, close to $p$, $F_1=F_0$ and the
invariance property implies that $F_1$ transfers $Q_s\cap W_0$ to
an open subset of a projective hyperplane. Now take a point
$a\subset W^z_0\subset \cx^{n-1}$ and consider an open connected
subset $V_a\subset V_1$, which consists of $q\in U_1$ such that
$(\{a\}\times W^w_0)\cap Q_q\neq\emptyset$. Clearly, each $V_a$ is
open, $V_a=\bigcup\limits_{b\in W^w_0}Q_{(a,b)}$ so that $V_a$ is
connected and $V_1=\bigcup\limits_{a\in W^z_0}V_a$. The set $V_a$
always contains points, sufficiently close to $p$, and we may
consider on $V_a$ the holomorphic map, which assigns to $q\in V_a$
the $k$-jet, $k\geq 2$, of $Q_q$ at the point $a$. The components
of this map, corresponsing to derivatives of order $\geq 2$ vanish
for points in $V_a$, close to $p$ (since for such points
$F_1(Q_q\cap W_0)$ is contained in a projective hyperplane), and
consequently vanish on entire $V_a$, so for all $q\in V_a$, $F_1$
transfers the connected component of $Q_q\cap W_0$, containing the
point with $z$-coordinates equal to $a$, to an open subset of a
projective hyperplane. From this and~\eqref{e:v1} it follows that
the desired $\Q$-Segre property holds for $F_1$ in $W_0$.

Now we take any $r\in Q_p$ and prove the $\Q$-Segre property for
$F_1$ is some polydisc $U(r)$. We choose $U(r)\subset W_1$ and so
that $Q_p\cap U(r)$ is the graph of a function over $U^z(r)$, so
that for $p^*$, sufficiently close to $p$, $Q_p\cap U(r)$ is
connected. Since $W_1\supset Q_p$, for $p^*$, close to $p$,
$Q_{p^*}\subset W_1$, and the $\Q$-Segre property of $F_1$ in
$W_0$ implies, by the uniqueness property, that $F_1(Q_{p^*})$ is
an open subset of a projective hyperplane, and so does
$F_1(Q_{p^*}\cap U(r))$. Now arguments analogous to those used
above show the  $\Q$-Segre property for $F_1$ in $U(r)$, which
completes Step 1.

\smallskip

{\bf Step 2.} We now prove (ii) and (iii) for $j=2$. This will
give us the base case for a general induction argument. The proof
for this case will be a small modification of the one in the
previous step.

From (i), $W_0\subset W_1$ and  $F_0 |_{W_0} =F_1 |_{W_0}$. We
choose a polydisc $U(p_1)\subset W_1$ with the $\Q$-Segre property
for $F_1$ and a connected neighbourhood $W_2$ of $Q_{p_1}$ such
that for each Segre variety $Q_q,\,q\in W_2$ the intersection
$Q_q\cap U(p_1)$ is the graph of a function over $U^z(p_1)$, in
particular, it is connected. Consider the set
$$
A_2=\{(Z,\zeta)\in W_2\times\CP{n}:\,F_1(Q_Z\cap U(p_1))\subset
Q'_{\zeta}\}.
$$
The $\Q$-Segre property of $F_1$ and arguments, similar to those
in \cite{Sh} show that $A_2$ is an nonempty complex analytic set
in $W_1\times\CP{n}$ of dimension $n$. By shrinking $W_2$ if
needed, we obtain in a similar fashion that $A_2$ defines a
locally biholomorphic mapping $F_2:\,W_2\longrightarrow\CP{n}$.
Since $Q_{p_1}\ni p$, we conclude that $p\in W_2$ and for points
$p^*\in W_2$, sufficiently close to $p$, $Q_{p_2}\subset W_1$ and
the intersection $W_0\cap W_{p^*}$ is connected. By the invariance
property of $F_1=F_0$ in $W_0$ we conclude that the point in $A_2$
over $p^*$ must equal $F_1(p^*)$, i.e.,
$F_2(p^*)=F_1(p^*)=F_0(p^*)$, which proves (ii) for $j=2$. The
proof of (iii) for this case follows the same pattern as in Step
1.

\smallskip

\noindent{\bf Step 3.} We now perform the induction step by
assuming that $j>2$ and that for all smaller
 $j$ the proposition holds, i.e., for all $k<j$ the desired extensions and polydiscs with the
 $\Q$-Segre property have been already obtained.  We treat the case $k=j$.

In the same way as in Step 2, we obtain, using the $\Q$-Segre
property of $F_{j-1}$, a polydisc $U(p_{j-1})$, a neighbourhood
$W_j$ of $Q_{p_{j-1}}$ with $Q_q\cap U(p_{j-1})$ connected for
$q\in W_j$ and a locally biholomorphic map
$F_j:\,W_j\longrightarrow\CP{n}$, corresponding to the
$n$-dimensional complex analytic set
$$
A_j=\{(Z,\z)\in W_j\times\CP{n}:\,F_{j-1}(Q_Z\cap
U(p_{j-1}))\subset Q_{\zeta}\}.
$$
To prove (ii), take now $Z$ close to $p_{j-2}$ so that $Q_Z$ is
contained in $W_{j-1}$. To clarify what $F_j(Z)$ equals, recall
that by assumption the proposition is proved for smaller $j$.
Therefore,
$$
F_{j-1}(Q_Z\cap U(p_{j-1}))=F_{j-3}(Q_Z\cap
U(p_{j-1}))=Q'_{F_{j-3}(Z)},
$$
so by the definition of $F_j$ and $F_{j-2}$ we obtain that
$F_j(Z)=F_{j-2}(Z)$, and (ii) is proved.

Finally, to prove (iii) for $F_j$ we take $r\in Q_{p_{j-1}}$ and a
polydisc $U(r)$ such that for $p^*$ close to $p_{j-1}$ the
intersection $Q_{p^*}\cap  U(r)$ is connected. We also may suppose
that $Q_{p^*}\subset W_j$. Since $F_{j-2}(Q_{p^*}\cap U(p_{j-2}))$
contained in a projective hyperplane and $F_{j-2}=F_{j}$ in
$U(p_{j-2})\subset W_j$, we conclude that $F_j(Q_{p^*})$ is
contained in a projective hyperplane. To obtain the entire
$\Q$-Segre property for $F_j$ we repeat the arguments from the
proof in Step 1. This completes the proof of the theorem.
\end{proof}

We now formulate the following corollary, which is a weaker form
of our main extension result, but is convenient for applications.

\begin{corol}
Let $M,X,p,U_1,U_2,F_0$ satisfy Proposition 4.1. Then for each
point $q\in U_1\setminus X$ there exists a connected path
$\gamma:\,[0,1]\longrightarrow U_1\setminus X$,
$\gamma(0)=p,\,\gamma(1)=q$ such that $F_0$ extends analytically
along $\gamma$ as a locally biholomorphic mapping to $\CP{n}$ and
for any $r\in \gamma$ there exists a polydisc $U(r)$, centred at
$r$ such that the mapping $F_r$ has the $\Q$-Segre property in
$U(r)$ (here $F_r$ is the element of the analytically continued
germ $F_0$ along $\gamma$ at the point $r$).
\end{corol}

\begin{proof} Proposition 3.1 implies that there exist points
$p_1, p_2, ..., p_{2j-1}\in U_1\setminus X$ such that $\,q\in
Q_{p_0,p_1,...,p_{2j-1}},\,j\geq 1$. We set $p_{2j}:=q$ and choose
connected paths $\Gamma_{0,2}\subset Q_{p_1},\,\Gamma_{2,4}\subset
Q_{p_3},...,\Gamma_{2j-2,2j}\subset Q_{p_{2j-1}}$ such that
\begin{multline}
\Gamma_{0,2}(0)=p_{0},\,\Gamma_{0,2}(1)=p_2,\,\Gamma_{2,4}(0)=p_2,\,\Gamma_{2,4}(1)=p_4,...,\\
\Gamma_{2j-2,2j}(0)=p_{2j-2},\,\Gamma_{2j-2,2j}(1)=p_{2j}=q.
\end{multline}
Then, applying Proposition 4.1, we conclude that $F_2$ is a local
biholomorphic extension of $F_0$ along $\Gamma_{0,2}$; $F_4$ along
$\Gamma_{2,4}$; ... ; $F_{2j}$ is a local biholomorphic extension
of $F_{2j-2}$ along $\Gamma_{2j-2,2j}$ (of course, we use the
connectivity and simple connectivity of the Segre varieties as
holomorphic graphs). Taking now $\gamma$ to be the union of the
paths of $\Gamma_{0,2},...,\Gamma_{2j-2,2j}$, we obtain the
desired local biholomorphic extension of $F_{0}$. The $\Q$-Segre
property for $F_r,\,r\in\gamma$ now follows from Proposition~4.1.
\end{proof}

\section{Extension along an arbitrary path}

In this section we prove that $F_0$ can be analytically continued
along {\it any} path in $U_1\setminus X$. We begin with the
following proposition.

\begin{propos}
Let $M,X,U_1,U_2$ satisfy Proposition 2.1, $r\in U_1\setminus X$,
$U(r)\subset U_1\setminus X$ a polydisc centred at $p$, and let
$F,G$ be two biholomorphic mappings $U(r)\longrightarrow\CP{n}$
with $\mathcal{Q}$-Segre property in $U(r)$. Then there exists a
linear automorphism $\tau$ of $\CP{n}$ such that $G=\tau\circ F$.
\end{propos}

\begin{proof}
Let $\lambda:\,U_1\longrightarrow\CC{N}$ and
$\lambda':\,\CP{n}\longrightarrow(\CP{n})^*$ be the Segre maps in
the preimage and the target domain respectively, $U'=F(U)$. We
consider the map $\tau:=G\circ F^{-1}$, which is a biholomorphic
map $U'\longrightarrow\CP{n}$ onto its image.

From the $\mathcal{Q}$-Segre properties of $F$ and $G$ we know
that $\tau$ maps ``many" (connected components of) intersections
of projective hyperplanes with $U'$ to open subsets of projective
hyperplanes, and want now to prove the same for the set of
hyperplanes, intersecting a ball $B$ in some coordinate chart in
$\CP{n}$. To do so, put $r':=F(r)$ and fix some coordinate ball
$B\subset U'$, centred at $r'$. Let $q\in Q_r$ so that $r\in
S_q$. Choose a polydisc $\tilde U(r)$ with the following
properties:

\smallskip

\noindent (i) $\tilde U(r)\subset U(r),\,F(\tilde U(r))\subset B$.

\smallskip

\noindent (ii) $S_q\cap \tilde U(r)$ is a graph over $\tilde
U^z(r)$ (we use the notation from Proposition 4.1).

\smallskip

According to property (ii), there exists a connected neighbourhood
$V(q)$ such that for each $\tilde r\in \tilde U(r)$ there exists
$\tilde q\in V(q)$ such that $\tilde r\in Q_{\tilde q}$ and
$Q_{\tilde q}\cap \tilde U(r)$ is connected (we simply use the
fact that $Q_{\tilde r}$ is close to $Q_r\ni q$). Choosing $\tilde
U(r)$ small enough, we may suppose the Segre map $\lambda$ is
injective in $V(q)$. Consider now the following mapping: taking
$\tilde q\in V(q)$, we associate $Q_{\tilde q}$ to it, then
consider $F(Q_{\tilde q}\cap \tilde U(r))$ - an open subset of a
projective hyperplane, and, using $\lambda'$, associate a point in
$(\CP{n})^*$ to it. This is an injective holomorphic map from
$V(q)$ to $(\CP{n})^*$, denote its image by $W'$. Consider also
the set of projective hyperplanes, intersecting $B$, and denote
this open connected set in $(\CP{n})^*$ by $A'$. Then $W'$ is an
open subset of $A'$ (by property (i)), and, by definition of
$\tau$, the map $\tau$ sends $l_{\mathbf a}\cap B$ with
$l_{\mathbf a}\in W'$ (here $l_{\mathbf a}$ is a projective
hyperplane that corresponds to $\mathbf a\in\CP{n}$) to open
subsets of projective hyperplanes. Considering, as in the proof of
Proposition~4.1, the high order jets of $\tau(l_{\mathbf a}\cap
B)$ as holomorphic mappings of $A'$, we see that their components,
corresponding to derivatives of order $\geq 2$, vanish for
$l_{\mathbf a}\in W$, so they must vanish for all $l_{\mathbf
a}\in A'$ and we obtain the desired "hyperplane-to-hyperplane"
property of $\tau$ for any hyperplane, intersecting $B$. As it can
be verified from many references (see, for example, \cite{tresse},
\cite{sukhov}, or \cite{merker1},\cite{merker2}) $\tau$ in this
case must be a local biholomorphic symmetry of the system of flat
second order complex differential equations
$$
y_{x_k x_l}=0,\ \ k, l=1, 2,..., n-1.
$$
Hence, it is a linear automorphism of $\CP{n}$, and $G=\tau\circ
F$, as required.
\end{proof}

We now can prove part (i) of Theorem~\ref{t.1}, which we formulate
in the following theorem.

\begin{thm}
Let $M\subset\CC{n}$ be a smooth real-analytic hypersurface,
$X\subset\CC{n}$ a complex hypersurface, $0\in X\subset M$.
Suppose that $M$ is Levi nondegenerate in $M\setminus X$ and
pseudospherical. Then there exists a neighbourhood $U_1$ of the
origin such that for each point $p\in (M\setminus X)\cap U_1$ any
local biholomorphic mapping
$F_0:\,(\CC{n},p)\longrightarrow\CP{n}$, transferring $(M,p)$ onto
an open piece of a nondegenerate real hyperquadric
$\mathcal{Q}\subset\CP{n}$, extends analytically along arbitrary
continuous path $\gamma:\,[0,1]\longrightarrow U_1\setminus X$,
$\gamma(0)=p$ as a local biholomorphic mapping into $\CP{n}$.
\end{thm}

\begin{proof} Let $U_1,U_2$ be a standard pair of neighbourhoods of the origin
such that $M$ is Segre-regular in $U_1$. Suppose on the contrary
that the claim of the theorem is false. Then, since for $t$ close
to $0$ the extension with $\mathcal Q$-Segre property already
exists, we can choose the smallest $t_0,\,0<t_0<1$, such that
$F_0$ does not extend analytically to $\gamma(t_0)$ along the path
$\gamma|_{[0,t_0]}$ with the $\mathcal{Q}$-Segre property in some
neighbourhood of each $\gamma(t),0\leq t\leq t_0$ ($t_0$ is simply
the supremum of $t$ such that $F_0$ extends to $\gamma(t)$ along
$\gamma|_{[0,t]}$ with the $\mathcal{Q}$-Segre property at each
point). Applying Corollary 4.2, we obtain a polydisc $U(r)$,
centred at $r=\gamma(t_0)$ and a mapping $\tilde F_r$ in $U(r)$
with the $\mathcal{Q}$-Segre property. We now take some $t^*$
close to $t_0$ with $t^*<t_0$ and $r^*=\gamma(t^*)\in U(r)$ and
denote the corresponding extension of $F_0$ with the
$\mathcal{Q}$-Segre property at some polydisc $U(r^*)$ by
$F_{r^*}$. Without loss of generality we may assume that
$U(r^*)\subset U(r)$. Then, applying Proposition 5.1 for the
polydisc $U(r^*)$ and mappings $\tilde F_r,F_{r^*}$ in it, we get
a linear automorphism $\tau$ of $\CP{n}$ such that
$F_{r^*}=\tau\circ \tilde F_r$ in $U(r^*)$. This equality clearly
implies, by the global nature of  $\tau$, that $\tau\circ \tilde
F_r$ is a holomorphic extension of $F_{r^*}$ to $U(r)\ni r$ with
the $\mathcal{Q}$-Segre property, which contradicts the definition
of $t_0$.
\end{proof}

Let $Y,Y'$ be complex manifolds, and $F_0:\,(Y,p)\longrightarrow
(Y',p')$ be a local biholomorphic mapping between them. Suppose
that the germ $(F_0,p)$ can be extended analytically along any
continuous path $\gamma\subset Y$, starting at $p$. By a \it
(multiple-valued) analytic mapping in the sense of Weierstrass \rm
we mean the collection  $\{(F_{q,\gamma},q)\}$ of all possible
germs, obtained by analytic extension of $(F_0,p)$ along all
possible continuous paths $\gamma\subset Y$, starting at $p$ and
ending at arbitrary points $q\in Y$ (see, for example,
\cite{shabat} for more details of this concept). If for an
arbitrary path $\gamma\subset Y$ the analytic extension of
$(F_0,p)$ along $\gamma$ gives the same element $(F_0,p)$, then
the (multiple-valued) analytic mapping simply determines a
holomorphic mapping $Y\longrightarrow Y'$ (\it the Monodromy
theorem). \rm Note that if $Y$ and $Y'$ are domains in $\CC{}$,
then this notion simply gives an accurate set-up for a
(multiple-valued) analytic function. Putting now $Y=U\setminus
X,\,Y'=\CP{n}$, we can formulate the following

\begin{corol}
Let $M,X,U_1,F_0,p$ satisfy Theorem 5.2. Then the mapping
$F_0:\,(\CC{n},p)\longrightarrow\CP{n}$ extends locally
biholomorphically to a (multiple-valued) analytic mapping
$F:\,U_1\setminus X\longrightarrow\CP{n}$ in the sense of
Weierstrass. Moreover, each analytic element $(F_r,r)$ of $F$ at a
point $r\in U_1\setminus X$ has the $\mathcal Q$-Segre property.
\end{corol}

The following examples, as well as Example~\ref{ex.nonsh} and the model
example $M^{log}$, illustrate behaviour of the map $F$. A special case
of Example~5.4 is considered in \cite{elz} and \cite{lamel}, and Example~5.5 is
borrowed from \cite{elz}.

\begin{ex}
Consider the standard hyperquadric
\begin{equation}\label{e.Q}
Q=\{(z^*,w^*)\in\CC{2}:\,\im w^*=|z^*|^2\},
\end{equation}
and the (multiple-valued) locally biholomorphic mappings
$F_{\alpha}:\,\CC{}\times(\CC{}\setminus\{0\})\longrightarrow\CP{2}$
given as
$$
z^*=zw^\alpha,\,w^*=w^{2\alpha},\,\alpha\in \RR{}\setminus \{0\}.
$$
Then it is not difficult to check, by plugging $F$ into the
defining equation of $Q$, that $F_\alpha^{-1}$ determined by
$-\frac{\pi}{2}<{\rm Arg \ w}<\frac{\pi}{2}$ maps $(Q,p^*)$,
$p*=(0,1)\in Q$ onto an open piece of the smooth real-analytic
hypersurface
\begin{equation}\label{e.Ma}
M_\alpha =\left\{(z,w)\in \cx^2: w=\bar
w(\sqrt{1-|z|^4}+i|z|^2)^{\frac{1}{\alpha}},\ \ |z|<1 \right\}.
\end{equation}
All $M_\alpha$ are nonminimal, as they contain $X=\{w=0\}$, and
Segre-regular in $|z|<1$. $F_\alpha$ turns out to be exactly the
(multiple-valued) locally biholomorphic mapping, provided by
Corollary~5.3. For $\alpha\in\mathbb{Z}$ the mapping $F_\alpha$ is
single-valued and extends holomorphically to $X=\{w=0\}$. Thus
$F_\alpha^{-1}$ performs a certain blow-up of the 3-sphere in
$\CC{2}$. For $\alpha$ rational the multiple-valued mapping
$F_\alpha$ is finitely-valued and extends to $X$ as a holomorphic
correspondence\,\cite{shaf} (the graph of $F_\alpha$ extends even
to an algebraic subset of $\CP{4}$ in this case). For irrational
$\alpha$ the mapping $F_\alpha$ is infinitely-valued and,
furthermore, the graph of a germ of $F_\alpha$ does not even
extend to a closed complex-analytic subset of $(U_1\setminus
X)\times\cx^{2}$ (note that such extension is possible for the
model example $M^{\rm log}$).
\end{ex}

\begin{ex} Considering the quadric $Q$ defined by \eqref{e.Q} and the blow-ups $G_m$
given by
$$
z^*=zw^{m},\ \ w^*=w,\ \ m\in\mathbb{Z},\ \ m>0.
$$
The image of $Q$ under $G_m$ is the union of algebraic hypersurfaces given
by
\begin{equation}\label{e.Km}
K_m = \{\im w=|z|^2|w|^{2m} \}.
\end{equation}
These are nonminimal with $X=\{w=0\}$ and Segre-regular in
appropriate polydiscs $U_m(0)$. Here $G_m$ are single-valued and
extend to $X$ holomorphically.
\end{ex}

\section{Application: transfer of sphericity}

The above continuation results imply the following remarkable fact
on the geometry of nonminimal real hypersurfaces. Throughout the
section we denote by $M^+$ and $M^-$ the connected components of
$M\setminus X$. The next theorem is reformulation of
Theorem~\ref{t.1}, part (ii).

\begin{thm}
Let $M\subset\CC{n}$ be a smooth real-analytic hypersurface,
containing a complex hypersurface $X\subset\CC{n}$ and Levi
nondegenerate in $M\setminus X$. Suppose that $M$ is
pseudospherical with $M^+$ being $(k,l)$-spherical. Then $M^-$ is
$(k',l')$ spherical with, possibly, $(k',l')\neq (k,l)$.
\end{thm}

\begin{proof}
We fix a standard pair of neighbourhoods $U_1,U_2$ such that $M$
is Segre-regular in $U_1$ and choose points $p^+\in M^+\cap U_1$
and $p^-\in M^-\cap U_1$ and a local biholomorphic map
$F_0:\,(\CC{n},p^+)\longrightarrow\CP{n}$ with
$F_+(M^+)\subset\mathcal{Q}$ for a nondegenerate hyperquadric
$\mathcal{Q}\subset\CP{n}$ of the signature $(k,l)$. Applying
Corollary 4.2, we can find a polydisc $U(p^-)$ and a local
biholomorphic map $F_-:\,U(p^-)\longrightarrow\CP{n}$ with
$\mathcal{Q}$-Segre property. Put $\mathcal{P}:=F_-(M^-\cap
U(p^-))$. Then $\mathcal{P}\subset\CP{n}$ is a smooth
real-analytic Levi nondegenerate hypersurface, biholomorphically
equivalent to $M^-\cap U(p^-)$. The $\mathcal{Q}$-Segre property
of $F_-$ and the holomorphic invariance of Segre varieties imply
that all Segre varieties of $\mathcal{P}$, in some neighbourhood
of $F_-(p^-)$, are open pieces of projective hyperplanes. Now
choose some affine chart, containing $\mathcal{P}$ and make an
invertible affine transformation such that in new coordinates
$\mathcal{P}$ has the form
$$
2\re w'=H(z',\bar z')+ O(2),\ \ z'\in\CC{n-1},w\in\CC{},
$$
where $H(z',\bar z')$ is a nondegenerate Hermitian form. Then
Segre varieties of $\mathcal{P}$ have the form
$$
w=-\bar b'+H(z,\bar a')+... .
$$
This equation determines a hyperplane for all sufficiently small
$a$ and $b$, which implies that all monomials in dots in fact
vanish, and therefore, $\mathcal{P}$ is  a nondegenerate real
hyperquadric.
\end{proof}

The following example shows that, surprisingly, the equality
$(k,l)=(k',l')$  does not hold in Theorem 6.1 even for algebraic
nonminimal hypersurfaces in $\CC{n},\,n>2$ (in particular, $M^+$
and $M^-$ may have \it different \rm signature of the Levi form).

\begin{ex}\label{ex.nonsh}
Let $Q=\{\im w^*=|z^*_1|^2+|z^*_2|^2\}\subset\CC{3}$ be a real
strictly pseudoconvex hyperquadric. Consider the following
"blow-up"\, map $F$:
$$z_1^*=z_1\sqrt{w},\,z^*_2=z_2w,\,w^*=w.$$ Choosing a connected neighbourhood
$D\subset Q$ of the point $(0,0,1)\in Q$ and the single-valued
biholomorphic branch of $F$, given by
$-\frac{\pi}{2}<\mbox{Arg}\,w^*<\frac{\pi}{2}$, it is
straightforward to check that $F^{-1}$ maps $D$ onto an open piece
of the smooth real-analytic nonminimal hypersurface
$$M=\left\{w=\bar w\frac{(i|z_1|^2+\sqrt{1-2i|z_2|^2\bar
w-|z_1|^4})^2}{(1-2i|z_2|^2\bar w)^2} \right\},$$ satisfying $\re
w>0$ and $z_1,z_2,w$ be small enough (one should rewrite the
equation of $Q$ in the new coordinates). It is easy to check that
$M$ is Levi nondegenerate outside $w=0$, so $M$ satisfies the
conditions of Theorem 6.1, and at the point $p^{+}\in
M^+,\,p^+=(0,0,\varepsilon),\varepsilon>0$ the Levi form is
positive definite, though at the point $p^-\in
M^-,\,p^-=(0,0,-\varepsilon),\varepsilon>0$ the Levi form has
eigenvalues of different signs. So $M^+$ is $(2,0)$-spherical,
though $M^-$ is $(1,1)$-spherical. In fact, the single-valued
biholomorphic branch of $F$ given by $\frac{\pi}{2}<\mbox{Arg
w}<\frac{3\pi}{2}$ maps the negative half $M^-$ of $M$ onto a
domain on the indefinite hyperquadric $Q^-=\{\im
w^*=-|z_1^*|^2+|z^*_2|^2\}$.
\end{ex}

Unlike the case $n\geq 3$, for $n=2$ all hyperquadrics in $\CP{2}$
are equivalent to the 3-sphere $S^3\subset\CC{2}$ and the
phenomenon from Example~6.2 can not hold. However, it may still
happen that the multiple-valued mapping, obtained in Theorem 5.2,
maps $M^+$ and $M^-$ to \it different \rm hyperquadrics in
$\CP{2}$, though these hyperquadrics are equivalent by means of
some $\tau\in\mbox{Aut}(\CP{2})$.

\begin{ex}
Consider the hypersurface $M^{\rm log}\subset\CC{2}$ (see
Introduction). Then the multiple-valued map
$F:\,(z,w)\longrightarrow(z,\ln w)$ maps the domain $M^+\subset
M$, given by the condition $u>0$, to the hyperquadrics
$$
\{\im w^*+2k\pi=|z^*|^2\}, \ \ k\in\mathbb{Z},
$$
and the domain $M^-\subset M$, given by $u<0$, to the
hyperquadrics
$$
\{\im w^*+(2k+1)\pi=|z^*|^2\}, \ \ k\in\mathbb{Z}.
$$
Each of the hyperquadrics looks as $\tau^k(\mathcal Q_0)$, where
$\mathcal Q_0$ is the standard hyperquadric $\{\im w^*=|z^*|^2\}$
and the element $\tau\in\mbox{Aut}(\CP{2})$ is the affine
transformation $(z,w)\longrightarrow (z,+\pi i)$.
\end{ex}

As a small consolation for the paradoxical phenomenon, illustrated
by Examples~6.2,~6.3, we show now that this does not happen if $F$
is single-valued.

\begin{propos}\label{p.kl}
Let $M\subset\CC{n}$ be a smooth real-analytic hypersurface,
containing a complex hypersurface $X\subset\CC{n}$ and Levi
nondegenerate in $M\setminus X$. Suppose that $M$ is
pseudospherical with $M^+$ being $(k,l)$-spherical and the
multiple-valued analytic mapping $F$, obtained in Theorem~5.2, is
single-valued. Then:

\smallskip

(i) $M^-$ is also $(k,l)$-spherical.

\smallskip

(ii) $F$ maps both components $M^+,M^-$ to the same hyperquadric
$\mathcal{Q}$.
\end{propos}

\begin{proof}
Choose $U_1,p^+,p^-,F_0,\mathcal{Q}$ as in the proof of Theorem
6.1 and apply Propositions 3.1 and 4.1 to find a sequence
$p_0=p^+,p_1,...,p_{2j-1}$ such that the point $p_-\in
Q_{p_{2j-1}}$ as well as all the continuations $F_1,...,F_{2j}$.
Since $F$ is single-valued,  the continuations are simply
restrictions of $F$ onto some domains in $U_1\setminus X$. By the
definition of $F_{2j}$ we have $F_{2j-1}(Q_{p^-})\subset
Q'_{F_{2j}(p^-)}$ so that $F(Q_{p^-})\subset Q'_{F(p^-)}$. But
$p^-\in M$ and accordingly $p^-\in Q_{p^-}$, so $F(p^-)\in
Q'_{F(p^-)}$, which means that $F(p^-)\in \mathcal{Q}$. Since
$p^-\in M^-$ is arbitrary, this shows that
$F(M^-)\subset\mathcal{Q}$ and proves both (i) and (ii).
\end{proof}

It was communicated to us by V.\,Beloshapka that an alternative
proof of Theorem 6.1 can be deduced from the differential
equations characterizing sphericity of a Levi nondegenerate
hypersurface obtained by J.\,Merker in
\cite{merker1},~\cite{merker2}.

\section{The monodromy}

In this section we give a description of the multiple-valued
extension obtained in Theorem~5.2. It will allow us to find an
interesting interaction between nonminimal pseudospherical
hypersurfaces in $\CC{n}$ and linear differential equations of
order $n$.

Let $M, X, U_1,p,F_0$ satisfy Theorem~5.2, and let $F$ be the
(multiple-valued) analytic mapping obtained there. Consider a
noncontractible cycle $\gamma:\,[0,1]\longrightarrow U_1\setminus
X$, $\gamma(0)=\gamma(1)=p$, which is a generator of the
fundamental group of $U_1\setminus X$. Let $(F_1,p)$ be the
analytic continuation of the element $(F_0,p)$ of $F$ along
$\gamma$ to the point $p$. Applying the $\mathcal{Q}$-Segre
property of $F_0,F_1$ and using Proposition~5.1, we obtain a
mapping $\sigma\in {\rm Aut}(\CP{n})$ such that $F_1=\sigma\circ
F_0$. General properties of analytic continuation and the global
character of $\sigma$ imply that the linear automorphism $\sigma$
is independent:

\smallskip

(i) Of the choice of a generator $\gamma$,

\smallskip

(ii) Of the choice of an analytic element $(q,F_{q,0})$ of $F$ at
a point $q\in U_1\setminus X$.

\smallskip

To show (ii), for example, we choose a path $\gamma_q$ such that
$F_{q,0}$ is an extension of $F_0$ along $\gamma_q$ and denote by
$F_{q,1}$ the extension of $F_{q,0}$ along $\gamma$ (we suppose,
without loss of generality, that $q\in\gamma$). Again, the
$\mathcal{Q}$-Segre property of the elements of $F$ and
Proposition~5.1 show that there exists an element $\sigma'\in {\rm
Aut}(\CP{n})$ such that $F_{q,1}=\sigma'\circ F_{q,0}$. Note that
$F_{q,1}$ is obviously the extension of $F_1$ along $\gamma_q$.
But $F_1=\sigma\circ F_0$ so that the extension of $F_1$ along
$\gamma_q$ equals (by the uniqueness) $\sigma\circ F_{q,0}$ and we
conclude that $\sigma\circ F_{q,0}=\sigma'\circ F_{q,0}$ and
finally $\sigma=\sigma'$, as required. The proof of (i) is
analogous.

To see the dependence of $\sigma$ on the choice of  the initial
local biholomorphic mapping $F_0$ of $M$ onto a hyperquadric,
choose some other local biholomorphic mapping $\tilde F_0$ of
$(M,p)$ to a possibly different hyperquadric $\mathcal{\tilde Q}$
and denote the continuation of $\tilde F_0$ along $\gamma$ by
$\tilde F_1$ and the corresponding linear automorphism of $\CP{n}$
by $\tilde\sigma$. Then applying Proposition~5.1 we conclude that
there exists $\tau\in {\rm Aut}(\CP{n})$ such that $\tilde
F_0=\tau\circ F_0$ and so the continuation of $\tilde F_0$ along
$\gamma$ equals $\tau\circ F_1=\tau\circ\sigma\circ F_0$. On the
other hand, $\tilde F_1=\tilde\sigma\circ\tilde
F_0=\tilde\sigma\circ\tau\circ F_0$ so that $\tau\circ\sigma\circ
F_0=\tilde\sigma\circ\tau\circ F_0$ and
$\tau\circ\sigma=\tilde\sigma\circ\tau$. In fact, the linear
automorphism $\tau$ is a linear projective equivalence of
$\mathcal{Q}$ and $\mathcal{\tilde Q}$. We finally may express
$\tilde\sigma$ as follows:
\begin{gather}\label{e.10}
\tilde\sigma=\tau\circ\sigma\circ\tau^{-1}.
\end{gather}
Relation \eqref{e.10} shows that the \it monodromy matrix \rm
$\sigma$ is defined up to matrix conjugation and scaling.  We will
call this conjugacy class \it the monodromy operator of $M$ \rm
and denote it by $\Sigma$. \rm This term is used in analogy with
the \it monodromy matrix \rm of a linear differential equation  of
order $n$ at a singular point\,\cite{ilyashenko}. The monodromy
operator does not depend on the choice of the cycle $\gamma$, the
point $p\in U_1\setminus X$, the element $F_p$ of $F$, the target
hyperquadric $\mathcal Q$, and the initial local biholomorphic
mapping $F_0$ of $M^+$ or $M^-$ into $\mathcal Q$, and is only a
characteristic of the holomorphic geometry of a nonminimal
pseudospherical real-analytic hypersurface. This geometry can be
also characterized, for example, by the Jordan normal form of
$\Sigma$, defined up to scaling of its diagonal part or,
alternatively, by the cyclic subgroup
$H=\{\sigma^k,\,k\in\mathbb{Z}\}\subset \mbox{Aut} (\CP{n})$
generated by $\sigma$, defined up to conjugation. Note that the
subgroup $H$ exactly determines all possible elements of $F$ at a
point $p\in U_1\setminus X$, and all the elements have the form
$$
F_{p,k}=\sigma^{k}\circ F_{p,0},\ \ k\in\mathbb{Z},
$$ where
$F_{p,0}$  is some fixed element. Both the (scaled) Jordan normal
form of $\Sigma$, and (the conjugacy class of) the subgroup
$H\subset\mbox{Aut} (\CP{n})$ precisely characterize the monodromy
of $F$ about $X$.

The analogy with differential equations goes even further. Choose
the local coordinate system in such a way that $X$ is given in
$U_1$ by the condition $w=0$. Consider the
$(n+1)\times(n+1)$-matrix $\sigma$ (defined up to scaling). We set
$$
A:=\frac{1}{2\pi i}\ln\sigma
$$
(we may choose any of the matrix logarithms), and consider in a
neighbourhood $U(p)$ of $p$ the mapping
$$
G_0:\,U(p)\longrightarrow\CP{n},\,G_0(z,w):=w^{-A}\cdot F_0(z,w).
$$
Here we understand $F_0(z,w)$ as the column of its $n+1$
homogeneous coordinates and by $w^{-A}$ we understand the
functional matrix exponent $e^{-A\ln w}$. The definition of $G_0$
does not depend on the choice of the uncertain factor of $\sigma$
since the uncertain factor clearly just scales the column,
representing $G_0$, and that does not change the element
$G_0(z)\in\CP{n},\,z\in U(p)$. Then $G_0$ extends along an
arbitrary path in $U_1\setminus X$, because $F_0$ and the
matrix-valued mapping $w^{-A}$ do, and determines a
(multiple-valued) analytic mapping $G$. Since the monodromy of
$w^{-A}$ is given by
$$
w^{-A}\longrightarrow \sigma^{-1}w^{-A}=w^{-A}\sigma^{-1},
$$
the monodromy of $G$ is given by
$$
G_0\longrightarrow w^{-A}\cdot\sigma^{-1}\cdot\sigma\cdot
F_0=w^{-A}\cdot F_0=G_0.
$$
Hence, by the Monodromy theorem\,\cite{shabat}, $G$ is a
single-valued holomorphic mapping, and we obtain the following
formula characterizing the multiple-valuedness of $F$:
$$
F=w^A\cdot G\qquad\mbox{(the Monodromy formula)},
$$
where $G$ is a holomorphic mapping $U_1\setminus
X\longrightarrow\CP{n}$.  Note that a very similar formula holds
for the monodromy of the fundamental matrix of solutions of a
linear differential equation of order~$n$,~\cite{ilyashenko}. The
Monodromy formula generalizes Examples 1.1, 5.3, 5.4, and gives a
local monodromy representation of an \it arbitrary \rm
multiple-valued extension of a local biholomorphic mapping from a
nonminimal real hypersurface to a quadric.

We summarize our arguments in the following theorem, which is the
expanded formulation of Theorem~\ref{t.2}.

\begin{thm}
Let $M,X,U_1,\mathcal{Q}$ satisfy Theorem 5.2, and $F$ be the
(multiple-valued) analytic mapping obtained there. Then there
exists an element $\sigma\in\mbox{Aut}\,(\CP{n})$ such that:

\smallskip

1. The monodromy of $F$ with respect to a generator $\gamma$ of
the fundamental group of $U_1\setminus X$ is given by
$$F_q\longrightarrow\sigma\circ F_q,$$ where $F_q$ is an arbitrary
element of $F$ at a point $q\in U_1\setminus X$. In particular,
the collection of all elements of $F$ at a point $q$ is given by
the natural action of the cyclic subgroup of
\rm$\mbox{Aut}\,(\CP{n})$\it, generated by $\sigma$, on a fixed
analytic element of $F$ at $q$.

\smallskip

2. All possible changes of the target hyperquadric
$\mathcal{Q}\subset\CP{n}$ and the local biholomorphic map $F_0$,
transferring $(M,p)$ to $\mathcal{ Q}$ transform the monodromy
matrix $\sigma$ by the formula
$$\sigma\longrightarrow\tau\circ\sigma\circ\tau^{-1},$$
where $\tau\in\mbox{Aut}\,(\CP{n})$, and thus generate the
monodromy operator $\Sigma$. The correspondence $$M\longrightarrow
J(M),$$ where $J(M)$ is the (scaled) Jordan normal form of
$\Sigma$, is only characterized by holomorphic geometry of $M$. In
particular, $J(M)$ is a biholomorphic invariant of $M$.

\smallskip

3. If the local coordinates $(z,w)$ at the origin are chosen in
such a way that $X=\{w=0\}$ then there exists a single-valued
holomorphic mapping $G:\,U_1\setminus X\longrightarrow\CP{n}$ such
that the following Monodromy formula holds: $$F=w^A\cdot G,$$
where $2\pi iA$ is a complex logarithm of the monodromy matrix
$\sigma$.
\end{thm}

\begin{rema}
If $\sigma$ is a scalar matrix, i.e., the monodromy operator
$\Sigma$ is the identity, we conclude, by the Monodromy
theorem\,\cite{shabat}, that the multiple-valued map $F$ is, in
fact, a single-valued locally biholomorphic mapping $F:
U_1\setminus X\longrightarrow\CP{n}$.
\end{rema}

\begin{ex}
For the hypersurfaces $M_\alpha\subset\CC{2}$ given
by~\eqref{e.Ma} with $\alpha\in\mathbb{Z}$ and the hypersurfaces
$K_m\subset\CC{2}$ given by~\eqref{e.Km}, the monodromy operator
is the identity, and the map $F$ is single-valued. For the
hypersurfaces $M_\alpha$ with $\alpha\notin\mathbb{Z}$ the
monodromy operator has a diagonal Jordan normal form:
$$
J(M_\alpha)=\mbox{diag}\,\{e^{2\pi i\alpha},e^{4\pi i\alpha},1\}.
$$
Thus, the Monodromy representation becomes
$$
F=\begin{pmatrix}w^\alpha & 0 & 0\\ 0 & w^{2\alpha} & 0\\ 0 & 0 &
1\end{pmatrix} \cdot\begin{pmatrix}z \\ 1 \\ 1\end{pmatrix}.
$$
Finally, for the model example $M^{\rm log}$ the (appropriately
scaled) Jordan normal form is given by
$$
J(M^{\rm log})=\begin{pmatrix} 1 & 0 & 0\\
0 & 1 & 2\pi i\\ 0 & 0 & 1\end{pmatrix}.
$$
Decomposing the matrix to the sum of a diagonal and a nilpotent
matrices and computing the logarithm, we get $A=\begin{pmatrix}0 &
0 & 0\\ 0 & 0 & 1\\0 & 0 & 0\end{pmatrix}$, and so the Monodromy
representation takes the form:
$$F=w^{A}\cdot\begin{pmatrix} z\\ 0\\
1\end{pmatrix}=\begin{pmatrix}1 & 0 & 0\\ 0 & 1 & \ln w\\0 & 0 &
1\end{pmatrix}\cdot\begin{pmatrix} z\\ 0\\ 1\end{pmatrix}.$$
\end{ex}

\section{The algebraic case}

In this section we show that if the nonminimal pseudospherical
hypersurface $M$ in the preimage is algebraic then the
multiple-valued extension $F$, obtained by Theorem 5.2, admits
certain holomorphic extension to $X$.

We start with preliminaries.  Let $M\subset\CC{n}$ be a smooth
real-analytic nonminimal hypersurface, containing a complex
hypersurface $X$ and Levi nondegenerate in $M\setminus X$. We
choose local coordinates $(z,w) \in\CC{n-1}\times \CC{}$, at the
origin in such a way that the complex hypersurface, contained in
$M$, is given by $X=\{w=0\}$, and $M$ is given locally by the
equation
$$\im w=(\re w)^m\Phi(z,\bar z,\re w),$$ where $\Phi(z,\bar z,\re w)$ is a real-analytic function
in a neighbourhood of the origin such that  $\Phi(z,\bar z,0)\ne
0$ identically, $\Phi=O(2)$ and $m$ is a positive integer (see
\cite{ber} for the existence of such coordinates). We  may further
consider the local "complex defining equation" (see \cite{ber},
\cite{merker2}, \cite{elz}) of the form
$$
w=\bar w\,\Theta(z,\bar z,\bar w),
$$ where $\Theta=1+O(2)$ is real-analytic. So finally we come to
the following defining equation for $M$:
\begin{gather}
w=\bar w\,e^{i\varphi(z,\bar z,\,\bar w)},
\end{gather}
where the complex-valued real-analytic function $\varphi$ in a
polydisc $U\ni 0$ satisfies  the condition $\varphi(z,\bar z,\bar
w)=O(2)$ and also the reality condition
\begin{equation}\label{e.real}
\varphi(z,\bar z,w\,e^{-i\bar\varphi(\bar z,z,w)})\equiv
\bar\varphi(\bar z,z,w),
\end{equation}
reflecting the fact that $M$ is a real hypersurface. In what
follows we call (13) \it the exponential defining equation \rm for
a nonminimal hypersurface $M$.

Generalizing the ideas in \cite{lamel}, consider in a sufficiently
small polydisc $\tilde U\ni 0$ the real-analytic subset
$$\tilde M=\{(z^*,w^*)\in\tilde U:\,w^*=\bar w^*e^{\frac{i}{k}\varphi(z^*,\bar
z^*,(\overline{ w^*})^k)}\},$$ containing the complex hypersurface
$\tilde X=\{w^*=0\}$, where $k\geq 2$ is an integer. It follows
from~\eqref{e.real} that $\tilde M$ is in fact a smooth
real-analytic hypersurface and that the mapping
\begin{equation*}
z^*= z,\,w^*=\sqrt[k] w,\ \
-\frac{\pi}{k}<\mbox{Arg}\,w<\frac{\pi}{k},
\end{equation*}
sends the half of $M$ satisfying $\re w>0$ into the half of
$\tilde M$ satisfying $\re w^*>0$.  The hypersurface $\tilde M$ is
called \it the $k$-root of $M$. \rm Since the inverse mapping
\begin{equation}\label{e.nu}
\nu:\,z=z^*,\,w=(w^*)^k
\end{equation}
is holomorphic in all of $\tilde U$ and locally biholomorphic in
$\tilde U\setminus \tilde X$, it maps the entire $\tilde M$ into
$M$, preserving the complex hypersurfaces. This means that $\tilde
M\setminus \tilde X$ is Levi nondegenerate as well.

\begin{thm}
Let $M,X,U_1,p,F_0$ satisfy Theorem 5.2. Suppose, in addition,
that $M$ is real-algebraic and let $F:\,U_1\setminus
X\longrightarrow\CP{n}$ be the (multiple-valued) holomorphic
extension obtained in Theorem 5.2. Then:

\begin{enumerate}

\item[(i)] If the mapping $F$ is single-valued then it extends to $X$
holomorphically  and $F(M)\subset\mathcal Q$.

\item[(ii)] If the mapping $F$ is multiple-valued then it extends to $X$
as a holomorphic correspondence.   Furthermore, if the coordinates
$(z,w)$ in $U_1$ are chosen in such a way that $X\cap U_1=\{w=0\}$
then $F$ admits the representation $$F(z,w)=\tilde
F(z,\sqrt[k]{w}), $$ where $\tilde
F:\,(\CC{n},0)\longrightarrow\CP{n}$ is a single-valued
holomorphic mapping, and $k\geq 2$ is an integer.

\item[(iii)] Let $D\subset U\setminus X,\,X\subset\partial D$ be a domain
where the multiple-valued mapping $F$ admits a single-valued
branch. Then $F|_D$ extends continuously to $D\cup X$.

\end{enumerate}
\end{thm}

\begin{proof}
Fix $p\in M$ and $F_0$ as in Theorem 5.2. By Webster's theorem
\cite{webster} the graph of the local biholomorphic mapping
$F_0:\,U(p)\longrightarrow\CP{n}$ lies in a complex algebraic set
$A\subset U_1\times\CP{n}$ of dimension $n$. Accordingly, the
graph $\Gamma_F$ of the extended mapping $F$ lies in $A$ as well.
Let $\tilde A$ be the irreducible component of $A$ containing
$\Gamma_F$, and let $\pi:\,\tilde A\longrightarrow U_1$ and
$\pi':\,\tilde A\longrightarrow \CP{n}$ be the natural
projections. Compactness of $\CP{n}$ implies that the projection
$\pi$ is proper, so by Remmert's theorem, $\pi(\tilde A)$ is a
complex-analytic subset in $U_1$, so $\pi(\tilde A)=U_1$. Consider
now the set
$$
E=\{q\in U_1:\,\mbox{dim}(\pi^{-1}(q))>0\}.
$$ Then $E$ is a complex-analytic subset in $U_1$ (see, e.g.,~\cite{Lo}), and $\mbox{dim}\,E<n-1$ because
otherwise $\pi^{-1}(E)$ becomes a complex-analytic subset in
$\tilde A$ of dimension $\geq n$. Therefore, $X\varsubsetneq E$
and we can find a point $o\in X$ such that some polydisc $U(o)$
does not contain points from $E$. To prove (i) we suppose that $F$
is single-valued and choose a projective hyperplane
$\Pi\subset\CP{n}$ such that $\Pi$ does not intersect the finite
set $\pi'(\pi^{-1}(o))\subset\CP{n}$. Choosing appropriate
coordinates in $\CP{n}$ we may assume that
$\Pi=\CP{n}\setminus\CC{n}$ and accordingly
$\pi'(\pi^{-1}(U(o)))\Subset\CC{n}$. By Riemann's theorem we
conclude now that $X$ is a removable singularity for $F|_{U(o)}$.
Thus, $F$ extends holomorphically to the complement of $E$. Since
$\dim E \le n-2$, it follows that $F$ extends holomorphically to
all of $U$ (see, e.g., \cite{gr}). The inclusion $F(M)\subset
\mathcal Q$ follows from the uniqueness.

To prove (ii) we note that, by  algebraicity of $F_0$, its
multiple-valued extension $F$ is in fact finite-valued which means
that there exists an integer $k\geq 2$ such that the extension of
the analytic element $(F_0,p)$ along the path $\gamma^k$, where
$\gamma$ is the generator of $\pi_1(U_1\setminus X$), does not
change this element. Choose now the coordinates $(z,w)$ in $U_1$
in such a way that $X=\{w=0\}$ and $p\in M^+$ and consider $\tilde
M$ - the $k$-root of $M$. It follows from the arguments above that
$\tilde M$ satisfies all the conditions for Theorem~5.2 and we may
also consider the multiple-valued mapping $\tilde F$,
corresponding to the pseudospherical hypersurface $\tilde M$. The
map $\nu$ given by~\eqref{e.nu} gives the relation between $M$ and
$\tilde M$, and thus shows that the monodromy of $\tilde F$ with
respect to the generator $\gamma$ of $\pi_1(\tilde U\setminus X)$
is simply the identity, since the monodromy of $F$ with respect to
$\gamma^k$ is the identity. We conclude that the map $\tilde F$ is
single-valued and from claim (i), $\tilde F$ extends to $\tilde X$
holomorphically, and the explicit formula for $\nu$ now implies
(ii).

For the proof of  (iii) (only the multiple-valued case is not
immediate) it is easy to see from~(ii) that for each $o=(z_0,0)\in
X$ the limit $\lim\limits_{(z,w)\rightarrow o} F|_D(z,w)=\tilde
F(z_0,0)$, which shows the continuity of  the glued map in $D\cup
X$. This completes the proof of the theorem.
\end{proof}

\remark{8.2} When $\mathcal Q$ is strictly
pseudoconvex and $F$ is single valued, the set $F(X)$ in the above
theorem becomes a connected locally analytic set in $\mathcal Q$
which implies that $F(X)$ consists of one point. Using the $k$-root
construction it is easy to verify from here that in the
multiple-valued case the cluster set of $X$ with respect to any
single-valued branch of $F$ in a domain $D\subset U\setminus
X,\,X\subset\partial D$ consists of exactly one point.



\end{document}